\documentclass{amsproc}
\usepackage{amsmath}
\usepackage{float}
\usepackage{amsfonts}
\usepackage{amssymb}
\usepackage{hyperref}
\hypersetup{colorlinks,
citecolor=black,
filecolor=black,
linkcolor=black,
urlcolor=black}
\usepackage{graphics,epstopdf}
\usepackage[pdftex]{graphicx}
\usepackage{newlfont}\newlength{\defbaselineskip}
\usepackage[left=1in,right=1in,top=1in,bottom=0.8in,footskip=0.25in]{geometry}
\usepackage{setspace}
\allowdisplaybreaks
\newtheorem{theorem}{Theorem}[section]
\newtheorem{example}{Example}[section]
\newtheorem{lemma}{Lemma}[section]
\newtheorem{definition}[theorem]{Definition}
\newtheorem{remark}{Remark}[section]

\numberwithin{equation}{section}

\newtheorem{corollary}{Corollary}[section]

\usepackage{fancyhdr}
\usepackage{graphics}
\usepackage{color}
\usepackage{rotating}
\usepackage{fancybox}

\begin{document}

\begin{center}
\title{Construction of  Sz$\acute{\text{a}}$sz-Mirakjan-type operators which preserve $a^x, a>1$}
\maketitle
{\bf Rishikesh Yadav$^{1,\star}$, Vishnu Narayan Mishra$^{2,\dag}$ }\\ 

$^{1}$Applied Mathematics and Humanities Department,
Sardar Vallabhbhai National Institute of Technology,  Surat, Surat-395 007 (Gujarat), India.

$^{2}$Department of Mathematics, Indira Gandhi National Tribal University, Lalpur, Amarkantak 484 887, Anuppur, Madhya Pradesh, India.\\
\end{center}
\begin{center}
$^\star$rishikesh2506@gmail.com,
$^\dag$vishnunarayanmishra@gmail.com
\end{center}

\begin{abstract}
In this paper, we introduce a new type of  Sz$\acute{\text{a}}$sz-Mirakjan operators, which preserve $a^x$, $a>1$ fixed and $x\geq 0$. We study uniform convergence of the operators by using some auxiliary result and also error estimation is given. 
The convergence of said operators are shown and analyzed by graphics, also in same direction we find better rate of convergence than Sz$\acute{\text{a}}$sz-Mirakjan operators by analyzing the graphics. Voronovskaya-type theorem is studied and a comparison is shown under sense of convexity with  Sz$\acute{\text{a}}$sz-Mirakjan operators.  In last section, modified sequence is constructed in the space of integral function.

\end{abstract}

\textbf{Keywords:} Sz$\acute{\text{a}}$sz-Mirakjan operators, King's operators, modulus of continuity, Voronovskaya-type theorem, convexity of  function.

\section{Introduction}
In 1941, Mirakjan \cite{GM} defined the operators $S_n:C_2[0,\infty)\rightarrow C[0,\infty)$ for any $x\in [0,\infty)$ and for any $n\in\mathbb{N}$ given by\\\\
\begin{equation}\label{1}
S_n(f;x)=\sum\limits_{k=0}^{\infty}s_{n,k}(x)f\left(\frac{k}{n}\right), 
\end{equation}
where $s_{n,k}(x)=e^{-nx}\frac{(nx)^k}{k!}$,~~~~ $x\in[0,\infty)$,\\\\
and 
\begin{center}
 $C_2[0,\infty)=\{f\in C[0,\infty):\underset{x\rightarrow 0}{\lim}\;\frac{f(x)}{1+x^2}\;\;\;\; \text {exists and is finite}\},$
 \end{center} which is a Banach space endowed with 
 \begin{center}
 $\|f\|=\underset{x\in[0,\infty)}{\sup}\frac{|f(x)|}{1+x^2}$.
 \end{center}
 
  The operators $S_n(f;x)$ are called Sz$\acute{\text{a}}$sz-Mirakjan operators, where $s_{n,k}(x)$ are Sz$\acute{\text{a}}$sz's  basis functions. They were extensively studied in 1950 by  Sz$\acute{\text{a}}$sz \cite{OS}.\\
 
 The operators $S_n$ have many more properties similar to classical Bernstein operators.  Both are positive and linear operators. Most of the approximating operators $L_n$ (say) preserve  $e_i(x)= x^i\;(i=0,1)$ i.e. $L_n(e_0)=e_0(x)\; \text{and}\; L_n(e_1)=e_1(x),~~ n\in\mathbb{N}$. These conditions hold for the Bernstein polynomials, the   Sz$\acute{\text{a}}$sz-Mirakjan operators, the Baskakov operators (see \cite{PN}, \cite{CB}, \cite{PP}, \cite{RA}) but  $L_n(e_2)\neq e_2,$ for any of these operators.
 
 In 2003 King \cite{JP} approached  a sequence   $\{V_n\}$ of linear positive operators which modify the Bernstein operator and approximate each continuous function on [0,1]. These operators preserve the test functions $e_0$ and $e_2$ and have better rate of convergence than classical Bernstein operators in $0\leq x\leq\frac{1}{3}.$  King's approach was further investigated by several authors. In \cite{MA} the authors investigated some approximation result on the Meyer-K\"{o}nig and Zeller type operators which preserve $x^2$.   \\ 
 In 2006, Morales et al.  \cite{DK} considered an operators $B_{n,\alpha}$ which fix  $e_0$ and $e_0+\alpha e_2,$ where $\alpha\in [0,\infty)$ and  after that some extension appeared in paper  \cite{HG}. In this paper the authors  assume $\tau_n=(B_n\tau)^{-1}\circ \tau$, where $\tau$ is any strictly continuous function defined on [0,1] such that $\tau(0)=0$ and $\tau(1)=1$, i.e., they studied  a sequence $V_n^{\tau}:C[0,1]\rightarrow C[0,1]$, defined by 
 \begin{center}
  $V_n^{\tau}=B_nf\circ\tau_n=B_nf\circ (B_n\tau)^{-1}\circ \tau.$
 \end{center}
 Here the operators $V_n^{\tau}f$ preserve $e_0$ and  $\tau$. \\
 
 Moreover a general extension is seen in paper \cite{JM} by considering operators $(B_{n,0,j}f)$ which fix $e_0$ as well as $e_j$ where $1<j\leq n$ and are defined as 
 \begin{center}
 $B_{n,0,j}f(x)=\sum\limits_{k=0}^{n}\left(\begin{array}{c}n\\ k\end{array}\right){x}^n {(1-x)^{n-k}}f\left(\left(\frac{k(k-1)\cdots(k-j+1)}{n(n-1)\cdots(n-j+1)}\right)^{\frac{1}{j}}\right),\;\;\;\ f\in C[0,1].$
 \end{center}
But in whole of the process these operators (all above) are defined on a finite interval. \\
 
 On the other hand,  Duman and  \"{O}zarslan \cite{OD} constructed Sz$\acute{\text{a}}$sz-Mirakjan-type operators which reproduce $e_0$, $e_2$ and having better error estimation  than  the classical Sz$\acute{\text{a}}$sz-Mirakjan operators.\\ 

In 2016, a modification of Sz$\acute{\text{a}}$sz-Mirakjan-type operators which reproduce $e_0$ and $e^{2ax},\; a>0,$ were constructed by  Acar et al. \cite{TA}  and they studied important properties related to these operators. They proved uniform convergence, order of approximation with the help of a certain weighted modulus of continuity, and also discussed some shape preserving properties of the defined operators. 

\section{Construction of the operators}
Let $u_n(x)$ be a continuous sequence of function defined on $[0,\infty)$, with $0\leq u_n(x)<\infty$,  then by (\ref{1}), we have 
\begin{equation}\label{b}
L_n^*(f;u_n(x))=e^{-nu_n(x)}\sum\limits_{k=0}^{\infty}\frac{(nu_n(x))^k}{k!}f\left(\frac{k}{n}\right),
\end{equation}
for every $f\in C[0,\infty)$ and each $x\in [0,\infty),~~~\text{where}~~n\in\mathbb{N}$.

The set $\{e_0,e_1,e_2\}$ is a  $K_+-$subset of $C_\rho[0,\infty)$ for $\rho\geq2$. This space $C_\rho[0,\infty)$ is isomorphic to $C[0,1]$, where $K_+$ is the Korovkin set  (see \cite{AF2} for details). \\

Now, if  $u_n(x)$ is replaced by $s_n(x)$  defined as 
\begin{equation}\label{a}
s_n(x)=\frac{x\log{a}}{(-1+a^{\frac{1}{n}})n},\;\;\;\;\;\ \forall\;\ x\geq 0,\;\ a>1,
\end{equation}  
then we have the following linear positive  operators:
\begin{equation}\label{s}
S_{n,a}^*(f;x)=\sum\limits_{k=0}^{\infty}a^{\left(\frac{-x}{-1+a^{\frac{1}{n}}}\right)}\frac{(x\log{a})^k}{(-1+a^{\frac{1}{n}})^kk!}f\left(\frac{k}{n}\right)
\end{equation}
where $f\in C_\rho[0,\infty),\;\;\ \rho>0$ and $x\geq 0.$\\

\begin{remark}
\em{We have 
\begin{eqnarray*}
s_n(x)=\frac{x\log{a}}{\left(-1+a^{\frac{1}{n}}\right) n},~~~~~~~~~~\forall~ x\geq 0,\;\ a>1.
\end{eqnarray*}
Now taking limit as $n\rightarrow \infty$, then
\begin{eqnarray*}
\underset{n\rightarrow\infty}{\lim}s_n(x)&=& \underset{n\rightarrow\infty}{\lim}\frac{x\log{a}}{\left(-1+a^{\frac{1}{n}}\right) n}\\&=& x.
\end{eqnarray*}
 i.e. the given sequence of functions (\ref{a}) converges to $x$, then equation (\ref{b}) reduces to the well known Sz$\acute{\text{a}}$sz-Mirakjan operators (\ref{1}).}
\end{remark}


\section{Auxiliary results}
In this section, we shall give some properties of the operators (\ref{b}) which we shall use in the proof of the main theorem.

\begin{lemma}\label{L1}
\em{Let $e_i(x)=x^i,\;\ i=0,1,2,3,4.$ Then for all $x\geq 0$, we have 
\begin{eqnarray*}
1.~S_{n,a}^*(e_0;x)&=&1,\\
2.~S_{n,a}^*(e_1;x)&=&\frac{x\log{a}}{\left(-1+a^{\frac{1}{n}}\right) n},\\
3.~S_{n,a}^*(e_2;x)&=& \frac{x\log{a}\left(-1+a^{\frac{1}{n}}+x\log{a}\right)}{\left(-1+a^{\frac{1}{n}}\right)^2n^2},\\
4.~S_{n,a}^*(e_3;x)&=&\frac{x\log{a}\left(\left(-1+a^{\frac{1}{n}}\right)^2+3\left(-1+a^{\frac{1}{n}}\right) x\log{a}+x^2(\log{a})^2\right)}{\left(-1+a^{\frac{1}{n}}\right)^3n^3},\\
5.~S_{n,a}^*(e_4;x)&=&\frac{x\log{a}\left(\left(-1+a^{\frac{1}{n}}\right)^3+7\left(-1+a^{\frac{1}{n}}\right)^2 x\log{a}+6\left(-1+a^{\frac{1}{n}}\right)(x\log{a})^2+(x\log{a})^3\right)}{\left(-1+a^{\frac{1}{n}}\right)^4 n^4}.
\end{eqnarray*}}
\end{lemma}
\begin{proof}
We have $f\in C[0,\infty)$ and for each $x\geq 0$, then 
\begin{eqnarray*}
1.~ S_{n,a}^*(e_0;x) &=& \sum\limits_{k=0}^{\infty}a^{\left(\frac{-x}{-1+a^{\frac{1}{n}}}\right)}\frac{(x\log{a})^k}{(-1+a^{\frac{1}{n}})^kk!}1\\&=& a^{\left(\frac{-x}{-1+a^{\frac{1}{n}}}\right)}a^{\left(\frac{x}{-1+a^{\frac{1}{n}}}\right)} \\&=&1.\\
2.~ S_{n,a}^*(e_1;x) &=& \sum\limits_{k=0}^{\infty}a^{\left(\frac{-x}{-1+a^{\frac{1}{n}}}\right)}\frac{(x\log{a})^k}{(-1+a^{\frac{1}{n}})^kk!}\frac{k}{n}\\&=& \frac{a^{\left(\frac{-x}{-1+a^{\frac{1}{n}}}\right)}}{n} \sum\limits_{k=0}^{\infty}\frac{(x\log{a})^k}{(-1+a^{\frac{1}{n}})^k(k-1)!}\\&=& \frac{a^{\left(\frac{-x}{-1+a^{\frac{1}{n}}}\right)}}{n} a^{\left(\frac{x}{-1+a^{\frac{1}{n}}}\right)} \frac{(x\log{a})}{(-1+a^{\frac{1}{n}})}\\ &=& \frac{(x\log{a})}{(-1+a^{\frac{1}{n}})n},
\end{eqnarray*}
similarly, the other equalities can be  proved. 
\end{proof}
\begin{lemma}\label{L2}
\em{Let $\lambda\geq 0$. Then we have 
\begin{eqnarray*}
S_{n,a}^*(e^{\lambda t};x)&=&\sum\limits_{k=0}^{\infty}a^{\left(\frac{-x}{-1+a^{\frac{1}{n}}}\right)}\frac{(x\log{a})^k}{(-1+a^{\frac{1}{n}})^kk!}e^{\frac{\lambda k}{n}}\\&=& a^{\frac{\left(-1+e^{\frac{\lambda}{n}}\right) x}{-1+a^{\frac{1}{n}}}}.
\end{eqnarray*}}
\end{lemma}

\section{Convergence theorems}
\begin{theorem}\label{T1}
\em{Let $a>1$ then for all $x\in[0,\infty)$ we have 
\begin{eqnarray*}
&1.&S_{n,a}^*(f,x)~ \text{is linear and positive on}~ C_B[0,\infty),\\
&2.&S_{n,a}^*(a^t,x)=a^x,\\
&3.&\underset{n\rightarrow\infty}{\lim}S_{n,a}^*(f,x)\rightarrow f~ \text{uniformly on}~ [0,b],\;\ b>0,~ \text{provided}~ f\in C_\rho[0,\infty),\;\rho\geq 2,
\end{eqnarray*}
where $ C_B[0,\infty)$ is the space of all continuous and bounded functions defined on $[0,\infty)$.
}
\end{theorem}
\begin{theorem}\label{T2}
\em{Let $K$ be subspace of the Banach lattice $C_2[0,\infty)$ defined as 
\begin{eqnarray*}
K=\{f\in C[0,\infty): \underset{n\rightarrow +\infty}{\lim}{f(x)}\;\;\; \text{is finite}\}.
\end{eqnarray*}
Then $\underset{n\rightarrow\infty}{\lim} S_{n,a}^*(f)\rightarrow f$ uniformly on $[0,\infty)$ provided $f\in K.$
}
\end{theorem}
\textbf{Proof of Theorem \ref{T1}}
1. By above operators (\ref{s}) we can see that  $S_{n,a}^*(f;x)$ is linear and positive.\\

2. Since we have the operators
\begin{eqnarray*}
S_{n,a}^*(f;x)&=&\sum\limits_{k=0}^{\infty}a^{\left(\frac{-x}{-1+a^{\frac{1}{n}}}\right)}\frac{(x\log{a})^k}{(-1+a^{\frac{1}{n}})^kk!}f\left(\frac{k}{n}\right)
\end{eqnarray*}
then
\begin{eqnarray*}
S_{n,a}^*(a^t;x)&=&\sum\limits_{k=0}^{\infty}a^{\left(\frac{-x}{-1+a^{\frac{1}{n}}}\right)}\frac{(x\log{a})^k}{(-1+a^{\frac{1}{n}})^kk!}a^{\frac{k}{n}}\\&=& a^{\left(\frac{-x}{-1+a^{\frac{1}{n}}}\right)}\sum\limits_{k=0}^{\infty}\frac{(xa^{\frac{1}{n}}\log{a})^k}{(-1+a^{\frac{1}{n}})^kk!}\\&=& a^{\left(\frac{-x}{-1+a^{\frac{1}{n}}}\right)}\exp\left(\frac{xa^{\frac{1}{n}}\log{a}}{-1+a^{\frac{1}{n}}}\right).
\end{eqnarray*}

3. Before proving the third result or Theorem (\ref{T2}), we wish to generalize some properties, since we have to prove uniform convergence of the operators over any compact set $[0,b],\; b>0$ and extensively on $[0,\infty)$, then for $b>0$ (fixed), consider the homomorphism $F_b:C[0,\infty)\rightarrow C[0,b]$ (lattice homomorphism) defined as $F_b(f)=f|_{[0,b]}$ for every $f\in C[0,\infty).$ Moreover, in this case we can see that
 \begin{eqnarray}\label{c'}
\underset{n\rightarrow\infty}{\lim}F_b( S_{n,a}^*(e_i))=F_b(e_i),~~~~ \forall \; i=0,1,2~~ \text{and}~~~ 
  \underset{n\rightarrow\infty}{\lim}F_b( S_{n,a}^*(a^t))=F_b(a^x),~~~~ a>1.
\end{eqnarray} 
 
Now, the well known Korovkin type property with respect to monotone operators ( Theorem 4.1.4 (vi) of [\cite{AF2}, p. 199]) is used. Let $H$ be a cofinal subspace of $C(Y)$, where $Y$ is a compact set. If  $S: C(Y)\rightarrow C_2[0,\infty)$ is a lattice homomorphism and if $\{L_n\}$ is a sequence of  linear positive  operators from $C(Y)$ into $C_2[0,\infty)$ such that $ \underset{n\rightarrow\infty}{\lim}L_n(h)=S(h)$,\; $\forall$  $h\in H$, then $ \underset{n\rightarrow\infty}{\lim}L_n(f)=f$ provided that $f$ belongs to the Korovkin closure of $H$. \\
 
Hence, by (\ref{c'}) and above properties, we have (3) of Theorem (\ref{T1}) but extensively to get uniform convergence on $[0,\infty)$  of the sequence  $\{S_{n,a}^*(f)\},$ we have to add some other  properties.\\
 
 $K$ is a subspace of the Banach lattice $C_2[0,\infty)$ endowed with the sup-norm.\\
 
 For a given $\alpha>0$, let us suppose that the function $f_{\alpha}(x)=a^{-\alpha x},\; (a>1\; \text{and}\; x\geq 0)$, then we have
\begin{eqnarray*}
S_{n,a}^*(f_\alpha;x)&=&\sum\limits_{k=0}^{\infty}a^{\left(\frac{-x}{-1+a^{\frac{1}{n}}}\right)}\frac{(x\log{a})^k}{(-1+a^{\frac{1}{n}})^kk!}a^{\frac{-\alpha k}{n}}\\&=&
a^{\left(\frac{-x}{-1+a^{\frac{1}{n}}}\right)}\sum\limits_{k=0}^{\infty}\frac{(xa^{\frac{-\alpha }{n}}\log{a})^k}{(-1+a^{\frac{1}{n}})^kk!}\\&=& a^{\left(\frac{-x}{-1+a^{\frac{1}{n}}}\right)}\exp\left(\frac{(xa^{\frac{-\alpha }{n}}\log{a})}{(-1+a^{\frac{1}{n}})}\right)\\&=& a^{\left(\frac{-x}{-1+a^{\frac{1}{n}}}\right)}a^{\left(\frac{xa^{\frac{-\alpha}{n}}}{-1+a^{\frac{1}{n}}}\right)}\\ &=& a^{\left(\frac{x(-1+a^{\frac{-\alpha}{n}})}{-1+a^{\frac{1}{n}}}\right)}.
\end{eqnarray*}  
On the other hand, if we consider $f_{\alpha}=\exp(-\alpha x),\;\ x\geq 0,$ then with the help of Lemma (\ref{L2}), we have
\begin{eqnarray*}
S_{n,a}^*(f_{\alpha};x)=a^{\frac{\left(-1+e^{\frac{-\alpha}{n}}\right) x}{-1+a^{\frac{1}{n}}}}.
\end{eqnarray*} 
 Observe that 
 \begin{eqnarray*}
 \underset{n\rightarrow \infty}{\lim}S_{n,a}^*(f_{\alpha})=f_{\alpha}~~~ \text{uniformly on}~~ [0,\infty).
\end{eqnarray*}  
 Hence using this limit and applying the Proposition  4.2.5-(7) of (\cite{AF2}, p.215), we can obtain Theorem (\ref{T2}). 
 
\section{Error estimation}
In this section, we will  compute the rate of convergence of the given operators (\ref{s}) and other equalities.\\

Let us define a function $\phi_x(t),\; x\geq0$ by $\phi_x(t)=(t-x)$ then by Lemma \ref{L1}, we can  get the following results:

\begin{lemma}\label{L3}\em{
For each $x\geq 0,$ we have:
\begin{eqnarray*}
1.~ S_{n,a}^*(\phi_x(t);x)&=&-\frac{1}{\left(-1+a^{\frac{1}{n}}\right)n}\left(x(-n+a^{\frac{1}{n}}n-\log{a})\right),\\
2.~ S_{n,a}^*(\phi_x^2(t);x)&=&\frac{1}{\left(-1+a^{\frac{1}{n}}\right)^2n^2}x\left(\left(-1+a^{\frac{1}{n}}\right)^2n^2x-\left(-1+a^{\frac{1}{n}}\right)(-1+2nx)\log{a}+x(\log{a})^2 \right),\\
3.~ S_{n,a}^*(\phi_x^3(t);x)&=&\left(\frac{x}{(-1+a^{\frac{1}{n}})^3n^3}\right)(({-1+a^{\frac{1}{n}}})^3n^3x^2-({-1+a^{\frac{1}{n}}})^2(1-3nx+3n^2x^2)\log{a}\\
&+& 3({-1+a^{\frac{1}{n}}})x(-1+nx)(\log{a})^2-x^2(\log{a})^3),\\
4.~ S_{n,a}^*(\phi_x^4(t);x)&=&\left(\frac{x}{(-1+a^{\frac{1}{n}})^4n^4}\right)(({-1+a^{\frac{1}{n}}})^4n^4x^3\\&&- ({-1+a^{\frac{1}{n}}})^3(-1+4nx-6n^2x^2+4n^3x^3)\log{a}\\&&+({-1+a^{\frac{1}{n}}})^2x(7-12nx+6n^2x^2)(\log{a})^2\\ &&- 2({-1+a^{\frac{1}{n}}})x^2(-3+2nx (\log{a})^3+x^3(\log{a})^4).
\end{eqnarray*}}
\end{lemma}
\begin{proof}
Since for each $x\geq 0$, we have
\begin{eqnarray*}
1.~S_{n,a}^*(\phi_x(t);x)&=& \sum\limits_{k=0}^{\infty}a^{\left(\frac{-x}{-1+a^{\frac{1}{n}}}\right)}\frac{(x\log{a})^k}{(-1+a^{\frac{1}{n}})^kk!}\left(\frac{k}{n}-x\right)\\&=& \frac{a^{\left(\frac{-x}{-1+a^{\frac{1}{n}}}\right)}}{n}\sum\limits_{k=0}^{\infty}\frac{(x\log{a})^k}{(-1+a^{\frac{1}{n}})^k(k-1)!}-xa^{\left(\frac{-x}{-1+a^{\frac{1}{n}}}\right)}\sum\limits_{k=0}^{\infty}\frac{(x\log{a})^k}{(-1+a^{\frac{1}{n}})^kk!}\\&=& \frac{a^{\left(\frac{-x}{-1+a^{\frac{1}{n}}}\right)}}{n}\frac{(x\log{a})}{(-1+a^{\frac{1}{n}})}a^{\left(\frac{x}{-1+a^{\frac{1}{n}}}\right)}- xa^{\left(\frac{-x}{-1+a^{\frac{1}{n}}}\right)} a^{\left(\frac{x}{-1+a^{\frac{1}{n}}}\right)}\\ &=& \frac{(x\log{a})}{(-1+a^{\frac{1}{n}})n}-x.\\
2.~S_{n,a}^*(\phi_x^2(t);x)&=& \sum\limits_{k=0}^{\infty}a^{\left(\frac{-x}{-1+a^{\frac{1}{n}}}\right)}\frac{(x\log{a})^k}{(-1+a^{\frac{1}{n}})^kk!}\left(\frac{k}{n}-x\right)^2\\&=& \frac{a^{\left(\frac{-x}{-1+a^{\frac{1}{n}}}\right)}}{n^2}\sum\limits_{k=0}^{\infty}\frac{(x\log{a})^kk^2}{(-1+a^{\frac{1}{n}})^kk!}-\frac{2xa^{\left(\frac{-x}{-1+a^{\frac{1}{n}}}\right)}}{n}\sum\limits_{k=0}^{\infty}\frac{(x\log{a})^kk}{(-1+a^{\frac{1}{n}})^kk!}\\&&+ x^2a^{\left(\frac{-x}{-1+a^{\frac{1}{n}}}\right)}\sum\limits_{k=0}^{\infty}\frac{(x\log{a})^k}{(-1+a^{\frac{1}{n}})^kk!}\\&=& \frac{a^{\left(\frac{-x}{-1+a^{\frac{1}{n}}}\right)}}{n^2}\left(\frac{(x\log{a})}{(-1+a^{\frac{1}{n}})}+\left(\frac{(x\log{a})}{(-1+a^{\frac{1}{n}})}\right)^2\right)a^{\left(\frac{x}{-1+a^{\frac{1}{n}}}\right)}\\&&- \frac{2xa^{\left(\frac{-x}{-1+a^{\frac{1}{n}}}\right)}}{n}\left(\frac{(x\log{a})}{(-1+a^{\frac{1}{n}})n}\right)a^{\left(\frac{x}{-1+a^{\frac{1}{n}}}\right)}+x^2\\&=& \frac{x}{\left(-1+a^{\frac{1}{n}}\right)^2 n^2}((-1+a^{\frac{1}{n}})\log{a}+x(\log{a})^2-2xn\log{a}(-1+a^{\frac{1}{n}})\\&&+ n^2x(-1+a^{\frac{1}{n}})^2 )\\&=&\frac{x}{\left(-1+a^{\frac{1}{n}}\right)^2n^2}(\left(-1+a^{\frac{1}{n}}\right)^2n^2x-\left(-1+a^{\frac{1}{n}}\right)(-1+2nx)\log{a}\\&&+x(\log{a})^2).
\end{eqnarray*}
Similarly,  the other identities can be proved.
\end{proof}

Let $f\in C_B[0,\infty)$  and $x\geq 0.$ Then the modulus of continuity of $f$ is defined to be 
\begin{eqnarray*}
\omega(f,\delta)=\underset{|t-x|\leq \delta}{\sup}|f(t)-f(x)|,\;\;\;\;\;\;\;\;\;\ t\in[0,\infty).
\end{eqnarray*}

Based on the modulus of continuity, we have the theorem:  
\begin{theorem}\em{
For every $f\in C_B[0,\infty)$, the space of all continuous and bounded functions defined on $[0,\infty)$, $x\geq 0$ and $n\in \mathbb{N}$, we have 
\begin{eqnarray*}
|S_{n,a}^*(f;x)-f(x)|\leq 2\omega(f,\delta_{n,x}),
\end{eqnarray*}
 where $\delta_{n,x}=\sqrt{l}$ and the value of $l$ is given by (2)  of Lemma \ref{L3}. } 
\end{theorem}
\begin{proof}
Let $f\in C_B[0,\infty)$ and $x\geq 0$, we have
\begin{eqnarray*}
|f(t)-f(x)|&\leq & \omega(f,\delta_{n,x})\left(\frac{|t-x|}{\delta_{n,x}}+1\right), 
\end{eqnarray*}
now applying the operators $S_{n,a}^*$ to both sides, we have
\begin{eqnarray*}
S_{n,a}^*|f(t)-f(x)|& \leq & \omega(f,\delta_{n,x})S_{n,a}^*\left(\frac{|t-x|}{\delta_{n,x}}+1\right)\\&\leq & \omega(f,\delta_{n,x})\left(\frac{S_{n,a}^*|t-x|}{\delta_{n,x}}+S_{n,a}^*1\right)\\&=& \omega(f,\delta_{n,x})\left(\delta_{n,x}^{-1}\; S_{n,a}^*|t-x|+1\right)\\ &\leq & \omega(f,\delta_{n,x})\left(\delta_{n,x}^{-1}\;\sqrt{S_{n,a}^*(t-x)^2}+1\right),~~(\text{Using the Cauchy-Schwarz Inequality})    ~ \\&=& \omega(f,\delta_{n,x})\left(\delta_{n,x}^{-1}\;\sqrt{l}+1\right),
\end{eqnarray*}
where $l$ is given by (2) of lemma (\ref{L3}).
\begin{eqnarray*}
&=&\omega(f,\delta_{n,x})\left(\delta_{n,x}^{-1}\;\delta_{n,x}+1\right),~~~~~~~~~~~~~~~ where~~ \delta_{n,x}=\sqrt{l}\\ &\leq & 2\omega(f,\delta_{n,x}).
\end{eqnarray*}
\end{proof}

\section{Example, Graphical approach and analysis}

In this section, we shall discuss the convergence properties by graphs of the defined operators (\ref{s}) to given functions, and analyze the rate of convergence of the defined operators. At last, we shall show that the rate of convergence of the defined operators is better than Sz$\acute{\text{a}}$sz-Mirakjan operators $S_n(f;x)$.

\begin{example}\label{Ex1} 
Let $f(x)=x(x-\frac{1}{3})(x-\frac{1}{4})$, choosing $a=1.5,~ n=1~ \text{to}~10~\text{and}\,n=100,~ 500$ then the convergence to the function $f(x)$ by the operators  $S_{n,a}^*(f;x)$ are  illustrated in the figures. In that order, we shall see, the convergence of the operators $S_{n,a}^*(f;x)$ to the function $f(x)$ for particular value of $n=10$ for different values of $a>1$.
\end{example}


\begin{figure}[htb]
 \centering 
  \includegraphics[width=0.32\textwidth]{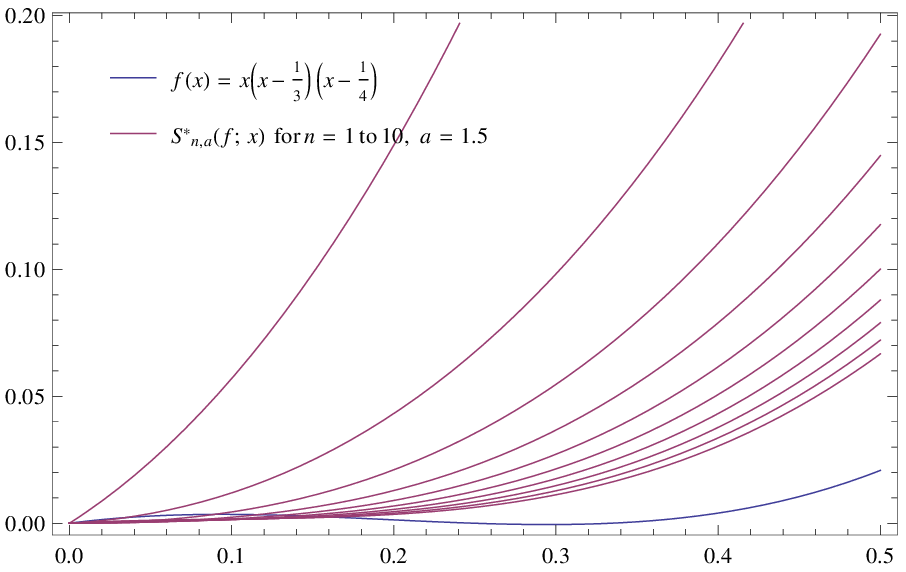}    \includegraphics[width=0.32\textwidth]{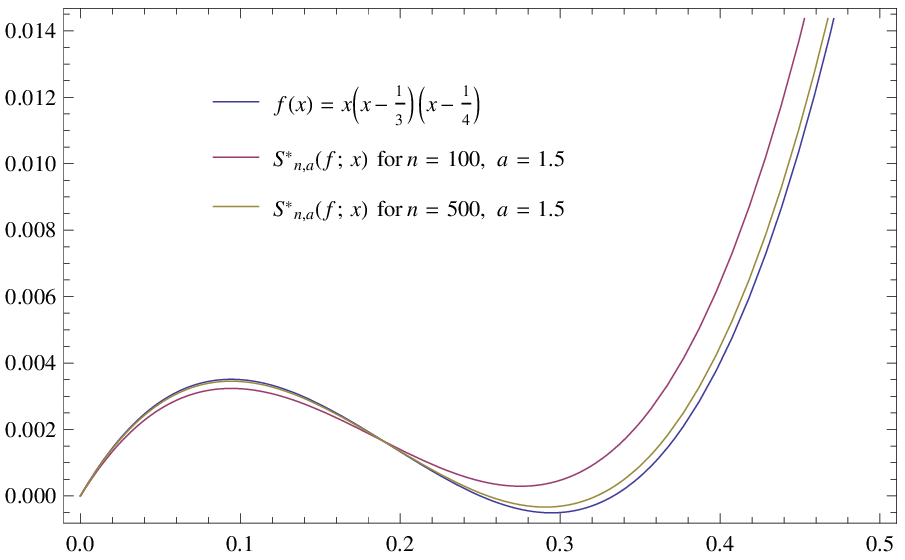}    \includegraphics[width=0.32\textwidth]{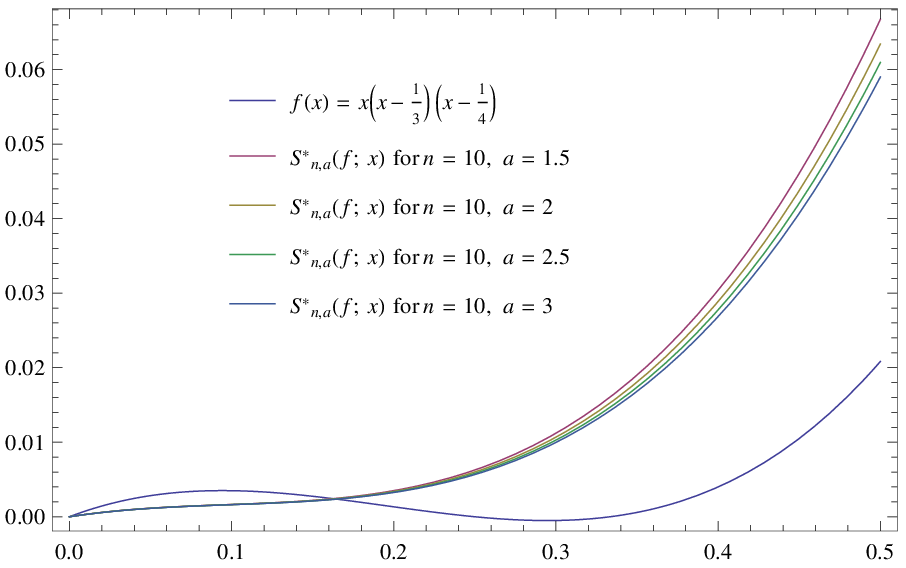} 
    \caption[Description in LOF, taken from~\cite{source}]{Convergence of the operators $S_{n,a}^*(f;x)$ to $f(x)$}
    \label{F1}
\end{figure}

\begin{figure}[htb]
 \centering 
   \includegraphics[width=0.4\textwidth]{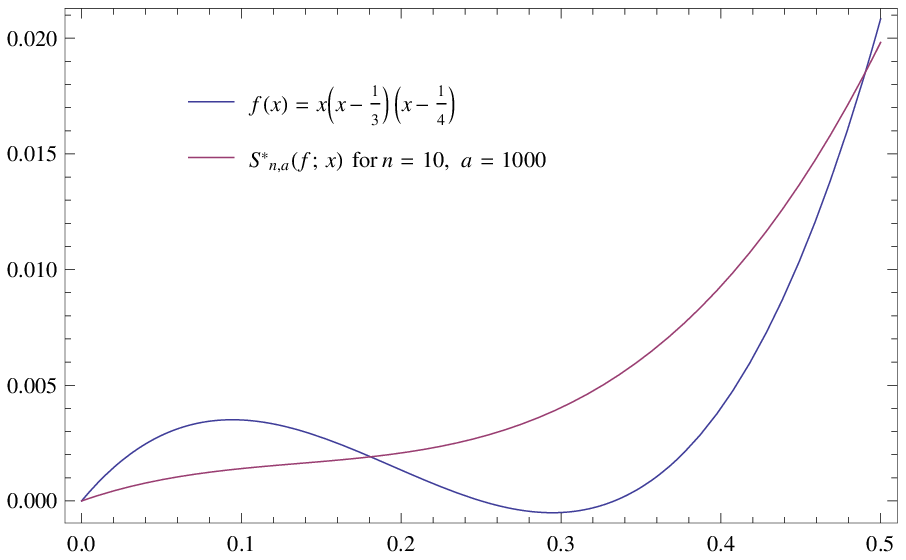}     \includegraphics[width=0.4\textwidth]{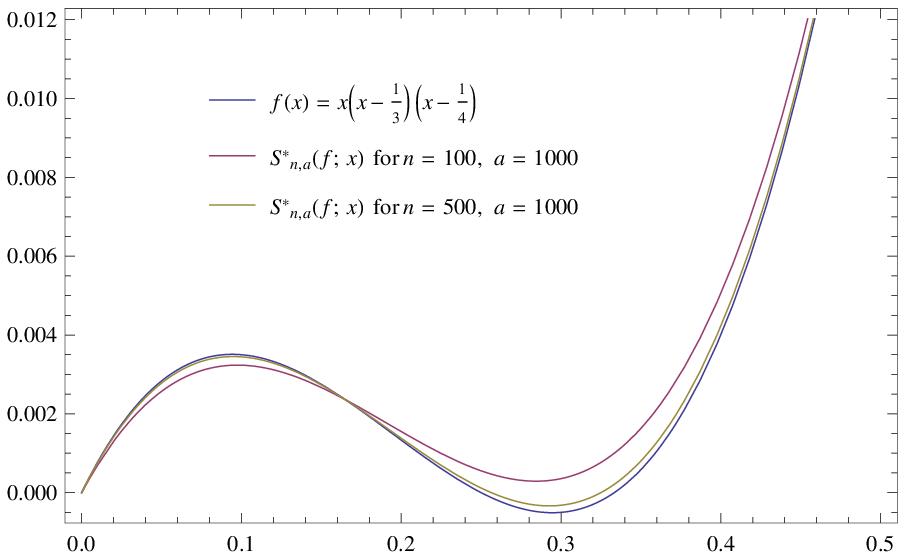} 
    \caption[Description in LOF, taken from~\cite{source}]{Convergence of the operators $S_{n,a}^*(f;x)$ to $f(x)$}
    \label{F2}
\end{figure}

\begin{example}\label{Ex2}
For $n=5,~30,~300$, the comparison of  convergence of the defined operators    $S_{n,a}^*(f;x)$ and Sz$\acute{\text{a}}$sz-Mirakjan operators   $S_n(f;x)$ to $f(x)=x^2\exp{4x}$ for fix value of $a$ are shown in figures. But for same function $f(x)$ the comparison is also take place by choosing different values of $a>1$ as $a=15,~150,~1500$.
\end{example}

\begin{figure}[htb]
    \centering 
  \includegraphics[width=0.24\textwidth]{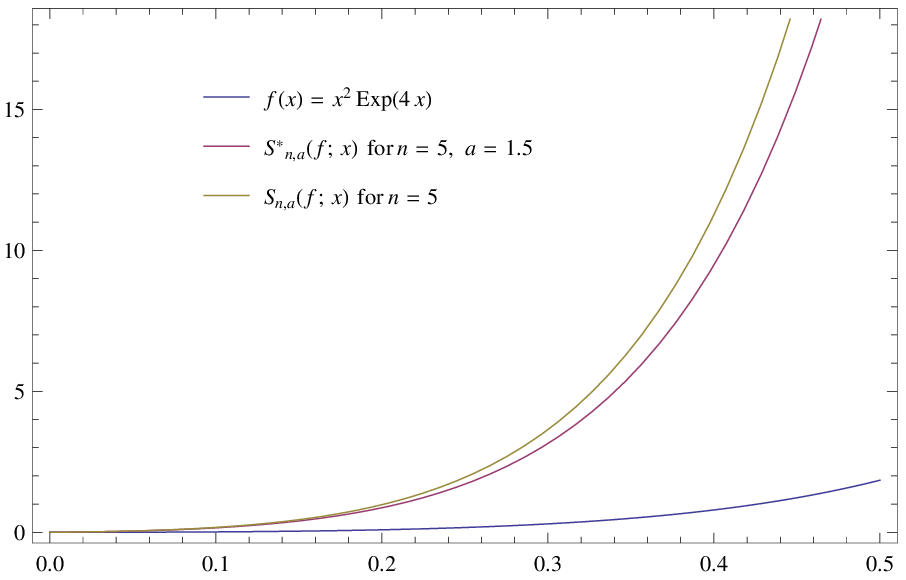}     \includegraphics[width=0.24\textwidth]{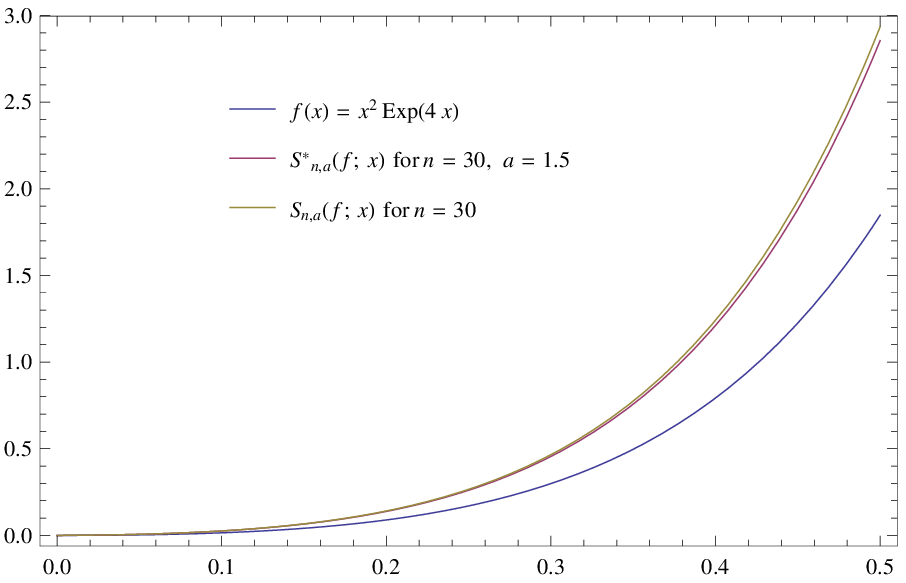}     \includegraphics[width=0.24\textwidth]{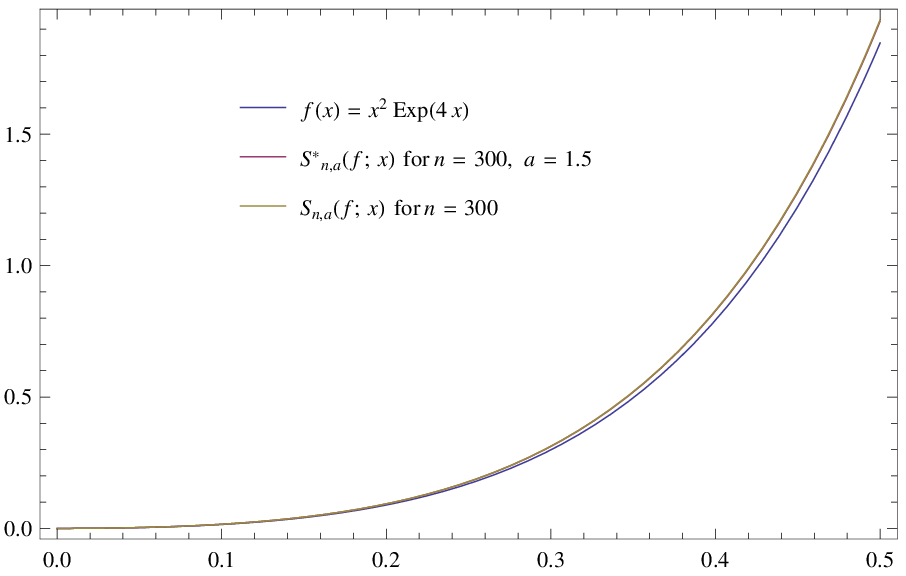}  
      \caption[Description in LOF, taken from~\cite{source}]{Comparison of the operators $\tilde{S}_{n,a}^*(f;x)$ and $S_n(f;x)$}
    \label{F3}
\end{figure}

\begin{figure} [h!]
    \centering 
   \includegraphics[width=0.32\textwidth]{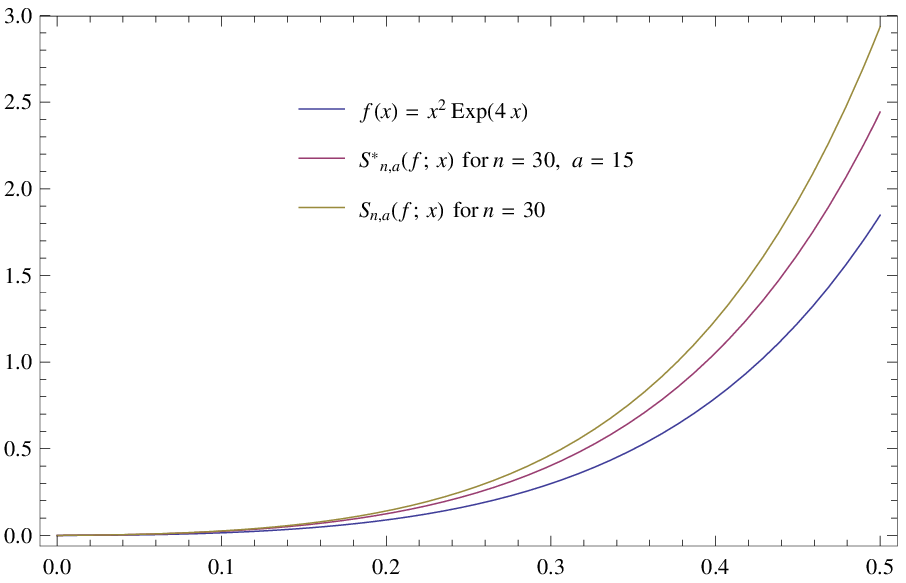}     \includegraphics[width=0.32\textwidth]{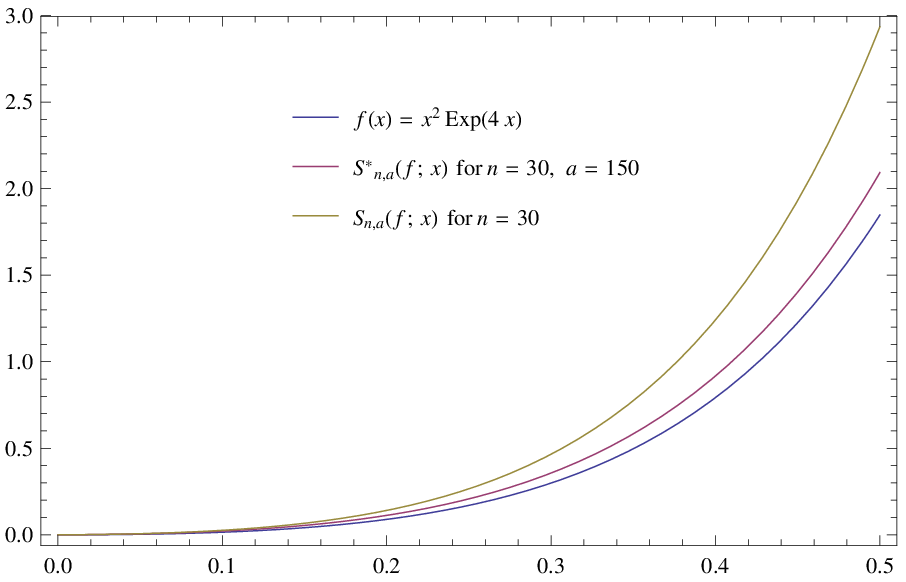}     \includegraphics[width=0.32\textwidth]{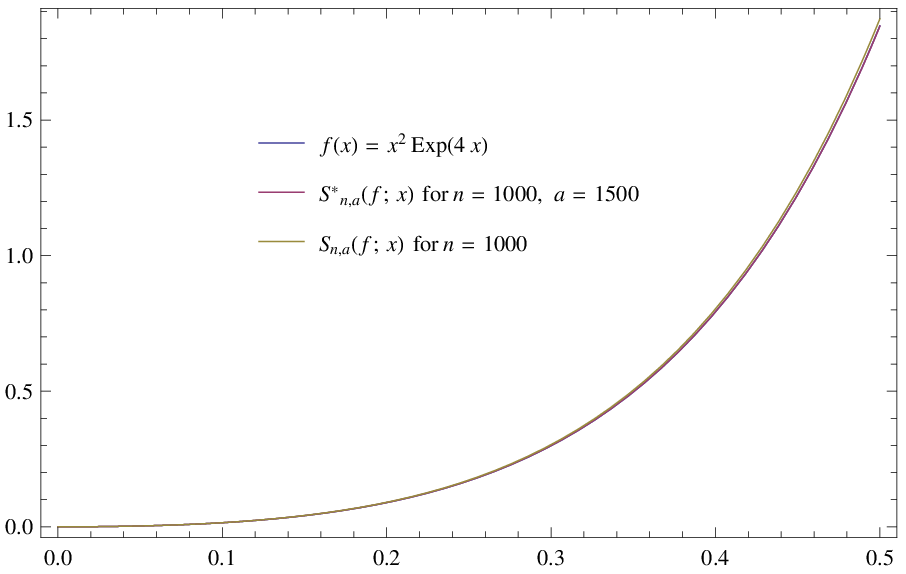}   
      \caption[Description in LOF, taken from~\cite{source}]{Comparison of the operators $\tilde{S}_{n,a}^*(f;x)$ and $S_n(f;x)$}
    \label{F4}
\end{figure}

\newpage
\begin{example}\label{Ex3}
Let $f(x)=x^2\cos{4x}$ and $a=1500$, for $n=10,~1000$, the comparison of  convergence of  $S_{n,a}^*(f;x)$ and Sz$\acute{\text{a}}$sz-Mirakjan operators  $S_n(f;x)$  are illustrated in figures.
\end{example}

\begin{figure} [h!]
    \centering 
  \includegraphics[width=0.4\textwidth]{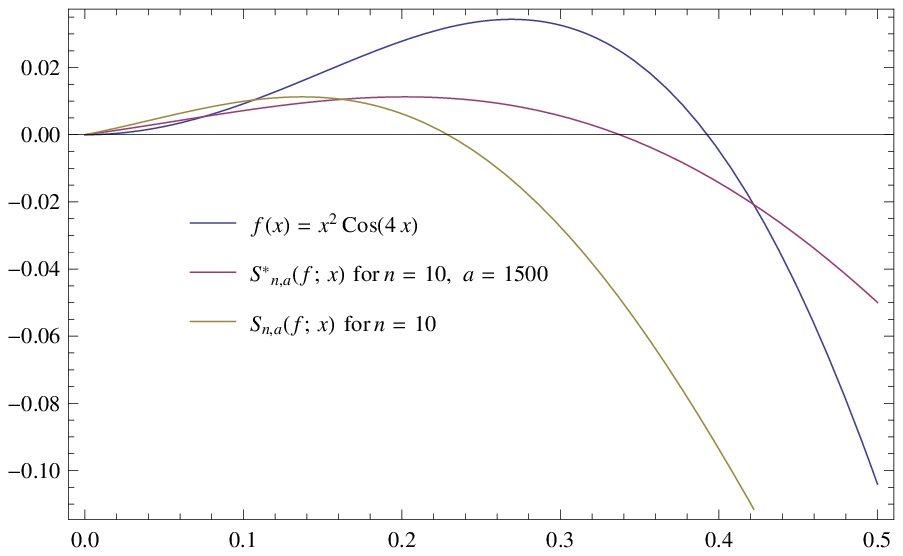}     \includegraphics[width=0.4\textwidth]{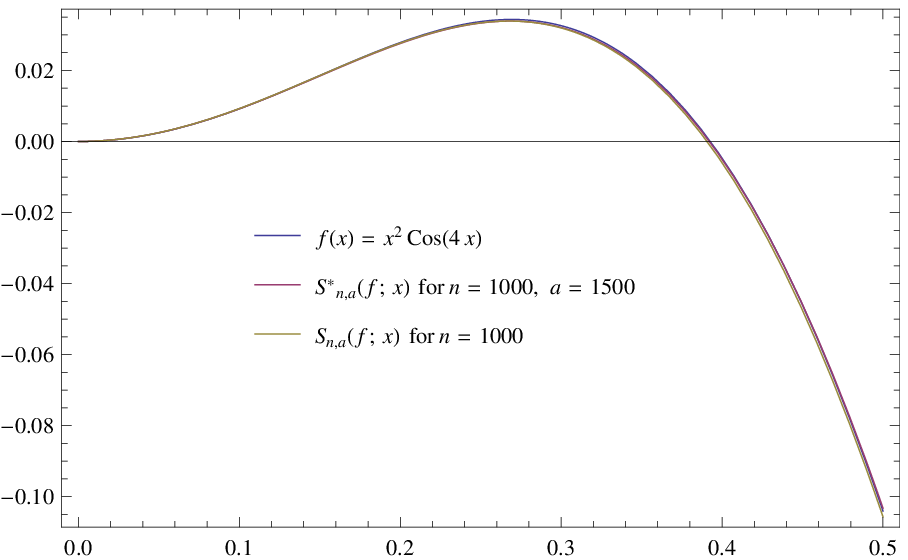}   
      \caption[Description in LOF, taken from~\cite{source}]{Comparison of the operators $S_{n,a}^*(f;x)$ and $S_n(f;x)$}
    \label{F5}
\end{figure}

\textbf{Conclusion}  The convergence of the operators $S_{n,a}^*(f;x)$ to the given function $f(x)$ is shown by above figures by taking different values of $n$ and fix $a$. The rate of convergence is also effected by taking different values of  $a$, it means rate of convergence is also depend on $a>1$, it can be observed by Figure \ref{F1}, additionally, in figure \ref{F2}, for fixing $a$ as $1000$, the convergence of the defined operators is took place. In Example  \ref{Ex2}, we put a comparison between $ S_{n,a}^*(f;x)$ and Sz$\acute{\text{a}}$sz-Mirakjan operators  $S_n(f;x)$ in sense of rate of convergence by both cases as one of them $a>1$ is fix and in other case (in Figure \ref{F4})  by taking different values of $a>1$ (even though fixed), we compute the rate of convergence and for large value of $n$ these operators converge to the function $f(x)$. For $n=10,~1000$, the rate  of convergence of the defined operators is good as Sz$\acute{\text{a}}$sz-Mirakjan operators  $S_n(f;x)$, it can be seen in above figures (\ref{F3}, \ref{F4}, \ref{F5}).

\section{A Voronovskaya-type theorem}
To examine the asymptotic behavior of the operators (\ref{s}), we will prove a  Voronovskaya-type theorem for the operators $S_{n,a}^*$ .  Before going to the main Theorem \ref{MT}, we have the lemma:

\begin{lemma}\em{
$\underset{n\rightarrow\infty}{\lim}n^2\;S_{n,a}^*(\phi_x^4(t);x)=3x^2,$ uniformly for each $x\in[0,b],\; b>0.$}
\end{lemma}
\begin{proof}
By Lemma \ref{L3}, we have
\begin{eqnarray*}
S_{n,a}^*(\phi_x^4(t);x)&=&\left(\frac{x}{(-1+a^{\frac{1}{n}})^4n^4}\right)(({-1+a^{\frac{1}{n}}})^4n^4x^3\\ &&-({-1+a^{\frac{1}{n}}})^3(-1+4nx-6n^2x^2+4n^3x^3)\log{a}\\ &&+({-1+a^{\frac{1}{n}}})^2x(7-12nx+6n^2x^2)(\log{a})^2\\ 
&&-2({-1+a^{\frac{1}{n}}})x^2(-3+2nx)(\log{a})^3+x^3(\log{a})^4)
\end{eqnarray*}
Using Mathematica, we get
\begin{eqnarray*}
\underset{n\rightarrow\infty}{\lim} n^2\;S_{n,a}^*(\phi_x^4(t);x)&=& \underset{n\rightarrow\infty}{\lim}n^2\left(\frac{x}{(-1+a^{\frac{1}{n}})^4n^4}\right)(({-1+a^{\frac{1}{n}}})^4n^4x^3\\
&&-({-1+a^{\frac{1}{n}}})^3(-1+4nx-6n^2x^2+4n^3x^3)\log{a}\\&&+({-1+a^{\frac{1}{n}}})^2x(7-12nx+6n^2x^2)(\log{a})^2\\
&&-2({-1+a^{\frac{1}{n}}})x^2(-3+2nx)(\log{a})^3+x^3(\log{a})^4)\\&=& 3x^2.
\end{eqnarray*}
\end{proof}
\begin{theorem}\label{MT} \em{
Let $f',\; f''\in  C_\rho [0,\infty)$, for some $f\in C_\rho[0,\infty),~~\rho\geq 4$, then we have
\begin{eqnarray*}
\underset{n\rightarrow\infty}{\lim}n\{S_{n,a}^*(f,x)-f(x)\}=\frac{-x}{2}(f'(x)\log({a})-f''(x))
\end{eqnarray*}
uniformly,  $x\in [0,b],\; b>0.$} 
\end{theorem}

\begin{proof}
We have $f',\;f''\in C_\rho[0,\infty)$ for some $f\in C_\rho[0,\infty)$ and $x\geq 0$. Now we define
\begin{eqnarray}\label{V1}
\xi(t,x)=\begin{cases} \frac{f(t)-f(x)-(t-x)f'(x)-\frac{1}{2}(t-x)^2f''(x)}{(t-x)^2}, & \text{if}\; t\neq x\\ 0, & \text{if}\;t\neq x\end{cases},
\end{eqnarray}

where $\xi(\cdot,x)\in C_\rho[0,\infty).$ By Taylor's theorem  we get
\begin{eqnarray}\label{V2}
f(t)=f(x)+(t-x)f'(x)+\frac{1}{2}(t-x)^2f''(x)+(t-x)^2~\xi(t,x)
\end{eqnarray}
Applying the operators $S_{n,a}^*$ to (\ref{V2}) and multiplying by $n$, we get
\begin{eqnarray}\label{V4}
n(S_{n,a}^*(f;x)-f(x))&=&nS_{n,a}^*((t-x);x)f'(x)+\frac{n}{2}S_{n,a}^*((t-x)^2;x)f''(x)\nonumber\\ &&+ nS_{n,a}^*((t-x)^2\xi(t,x))\nonumber \\ &=& nS_{n,a}^*(\phi_x(t);x)f'(x)+\frac{n}{2}S_{n,a}^*(\phi_x^2(t);x)f''(x)\nonumber\\ &&+ nS_{n,a}^*(\phi_x^2\xi(t,x)),
\end{eqnarray}
where $\phi_x(t)=(t-x)$.\\

Now applying the Cauchy-Schwartz	 theorem to $S_{n,a}^*(\phi_x^2\xi(t,x))$, we get 
\begin{eqnarray}\label{V5}
n|S_{n,a}^*(\phi_x^2(t)\xi(t,x);x)|\leq (n^2S_{n,a}^*(\phi_x^4(t);x))^{\frac{1}{2}}\;(S_{n,a}^*(\xi^2(t,x);x))^{\frac{1}{2}}.
\end{eqnarray}
Let $\xi^2(t;x)=\zeta(t;x)$ and $\zeta(x,x)=0$ as $\xi(x,x)=0$ and       $\zeta(\cdot,x)\in C_\rho[0,\infty)$. But we know that (already proved)\\

$\underset{n\rightarrow\infty}{\lim}S_{n,a}^*(f;x)=f(x)$ uniformly for each $x\in [0,b]$ and $b>0.$\\

So,\begin{eqnarray*}
\underset{n\rightarrow\infty}{\lim}S_{n,a}^*(\xi^2(t,x);x)=\underset{n\rightarrow\infty}{\lim}S_{n,a}^*(\zeta(t,x);x)=\zeta(x,x)=0.
\end{eqnarray*}
 By inequality (\ref{V5}), we get 
\begin{eqnarray}
\underset{n\rightarrow\infty}{\lim}n\;S_{n,a}^*(\phi_x^2(t)\xi(t,x);x)=0.
\end{eqnarray}
Now by (\ref{V4}), we have 
\begin{eqnarray*}
\underset{n\rightarrow\infty}{\lim}n\;(S_{n,a}^*(f;x)-f(x))&=&\underset{n\rightarrow\infty}{\lim}\{n\;S_{n,a}^*(\phi_x(t);x)f'(x)\}+\underset{n\rightarrow\infty}{\lim}\{\frac{n}{2}\;S_{n,a}^*(\phi^2_x(t);x)f''(x)\}\\ &=& \underset{n\rightarrow\infty}{\lim}n\frac{\left(-x(-n+a^{\frac{1}{n}}n-\log{a})\right)}{({-1+a^{\frac{1}{n}}})n} f'(x)\\ &+& \underset{n\rightarrow\infty}{\lim}\frac{n}{2}\left(\frac{x\left(\left(-1+a^{\frac{1}{n}}\right)^2n^2x-\left(-1+a^{\frac{1}{n}}\right)(-1+2nx)\log{a}+x(\log{a})^2 \right)}{({-1+a^{\frac{1}{n}}})^2n^2}\right) f''(x)\\&=& I_1f'(x)+I_2f''(x).
\end{eqnarray*}
Where,
\begin{eqnarray*}
I_1&=& \underset{n\rightarrow\infty}{\lim}n\frac{\left(-x(-n+a^{\frac{1}{n}}n-\log{a})\right)}{({-1+a^{\frac{1}{n}}})n}\\&=&  \underset{n\rightarrow\infty}{\lim}\frac{\left(-x(-1+a^{\frac{1}{n}}-\frac{\log{a}}{n})\right)}{({-1+a^{\frac{1}{n}}})\frac{1}{n}}\\&=&  \underset{p\rightarrow 0}{\lim}\frac{\left(-x(-1+a^{p}-p\log{a})\right)}{({-1+a^{p}})p},~~~~~~~\text{on replacing}~ \frac{1}{n}~ \text{by}~ p.
\end{eqnarray*}
Using L. hospital rule (as $\frac{0}{0}$ form) two times, we have 
\begin{eqnarray*}
I_1&=&  \underset{p\rightarrow 0}{\lim}\frac{-xa^p(\log{a})^2}{a^p(\log{a})^2p+2a^p\log{a}}\\&=& (-\frac{1}{2}x\log{a}),
\end{eqnarray*}
and
\begin{eqnarray*}
I_2&=& \underset{n\rightarrow\infty}{\lim}\frac{n}{2}\left(\frac{x\left(\left(-1+a^{\frac{1}{n}}\right)^2n^2x-\left(-1+a^{\frac{1}{n}}\right)(-1+2nx)\log{a}+x(\log{a})^2 \right)}{({-1+a^{\frac{1}{n}}})^2n^2}\right) \\&=& \underset{n\rightarrow\infty}{\lim}\frac{1}{2}\left(\frac{x\left(\left(-1+a^{\frac{1}{n}}\right)^2x-\left(-1+a^{\frac{1}{n}}\right)(-\frac{1}{n^2}+\frac{2x}{n})\log{a}+\frac{x}{n^2}(\log{a})^2 \right)}{({-1+a^{\frac{1}{n}}})^2\frac{1}{n}}\right) \\&=& \underset{l\rightarrow 0}{\lim}\frac{1}{2}\left(\frac{x\left(\left(-1+a^{l}\right)^2x-\left(-1+a^{l}\right)(-l^2+2xl)\log{a}+xl^2(\log{a})^2 \right)}{({-1+a^{l}})^2l}\right),~~~~~~~~~\text{on replacing}~ \frac{1}{n}~ \text{by}~ l.
\end{eqnarray*}
Using L. hospital rule (as $\frac{0}{0}$ form) three times, we have 
\begin{eqnarray*}
 I_2&=&\underset{l\rightarrow 0}{\lim}\frac{x\left(6a^{2l}x(\log{a})^3+(-1+a^l)xa^l(\log{a})^3-a^l(\log{a})^4(-l^2+2lx)-3a^l(\log{a})^3(-2l+2x)+6a^l(\log{a})^2         \right)}{2\left(6a^{2l}(\log{a})^3l+6a^{2l}(\log{a})^2+6(-1+a^l)a^l(\log{a})^2+2(-1+a^l)la^l(\log{a})^3\right)}\\&=& \frac{x}{2}
\end{eqnarray*}
Using $I_1$ and $I_2$ in above equation, we have
\begin{eqnarray}
\underset{n\rightarrow\infty}{\lim}n\;(S_{n,a}^*(f;x)-f(x))=-\frac{x}{2}(f'(x)\;\log{a}-f''(x))
\end{eqnarray}
\end{proof}
But for an unbounded interval, we generalize the above theorem by a corollary, given bellow :
\begin{corollary}\label{C2}
\em{For each $f\in C[0,\infty)$ such that $f,\; f',\; f''\in K$,  we have 
\begin{eqnarray*}
\underset{n\rightarrow\infty}{\lim}n\;(S_{n,a}^*(f;x)-f(x))=-\frac{x}{2}(f'(x)\;\log{a}-f''(x)),
\end{eqnarray*}
 holds uniformly for all $x\in [0,\infty).$ } 
\end{corollary}

\section{Comparison of new operators $S^*_{n,a}f$ with the classical Sz$\acute{\text{a}}$sz-Mirakjan operators with respect to convexity}

A comparison of the new operators given by (\ref{s}) with the   Sz$\acute{\text{a}}$sz-Mirakjan operators will take place. And we will show that the present operators (\ref{s}) have better rate of convergence under  certain conditions such as generalized convexity. 

\begin{definition}
\em{ A function $f(x)$ defined on $(a,b)$ is said to be convex with respect to $(e_0,\sigma_a(x))$  i.e.  $(1,\sigma_a(x))-$convex  provided
\begin{eqnarray}\label{c}
\left| \begin{matrix}
   1 & 1 & 1  \cr 
   \sigma_a(x_1) & \sigma_a(x_2) & \sigma_a(x_3)  \cr 
   f(x_1) & f(x_2) & f(x_3)  \cr 
\end{matrix}  \right|\geq 0,\;\;\;\;\ a<x_1<x_2<x_3<b,
\end{eqnarray} 
 
Note that if such an $f(x)\in C[a,b],$ then above ineqaulity will hold, by continuity, for all $a\leq x_1\leq x_2\leq x_3\leq b.$\\
The set of all functions which satisfy (\ref{c}) is denoted by $\ell(1,\sigma_a(x)).$ \\

In general for strict convexity i.e. a function $f$ is strictly $(1, \sigma_a(x))$ convex if above inequality is strictly greater than 0.}
\end{definition}

\begin{remark}\label{R2}
\em{A function $f\in C^2[0,\infty)$ is convex with respect to the function $\sigma_a(x)=a^x,\;a>1,$ iff 
\begin{eqnarray*}
f''(x)\geq (\log({a})) f'(x),\;\;\;\;\ x\geq 0.
\end{eqnarray*}}
\end{remark}

\begin{proof}
Using Remark (3.1) of \cite{MB}, for which a function  $f\in C^2[0,1]$ is convex with respect to any continuous and strictly increasing function $\tau_1(x)$ iff
\begin{eqnarray*}
f''(x)\geq\frac{\tau_1''(x)}{\tau_1'(x)}f'(x),~~~~~~~~x\in[0,1],
\end{eqnarray*}
but extensively on infinite interval, Acar et al. \cite{TA} shown a relation of the given function $f\in C^2[0,\infty)$ with respect to exponential function.\\

Now replacing $\tau_1(x)$ by $a^x$, we have the above result.
\end{proof}

By Corollary \ref{C2} and Remark \ref{R2}, we have       
\begin{corollary}
\em{If given $f\in C^2[0,\infty)$ is $(1,\sigma_a(x))$-convex or $\sigma_a(x)$-convex  with respect to $\sigma_a(x)=a^x,\; a>1$, for all $x\in [0,\infty)$ then $\exists \;  N\in \mathbb{N},$ so that we have 
\begin{eqnarray*}
f(x)\leq S_{n,a}^*(f;x),~~~~~ n\geq N.
\end{eqnarray*}
Here $N$ is dependent on $x$.}
\end{corollary}
For main theorems, we have following theorem of Cheney and Sharma  \cite{EW} (also see  \cite{DD}).

\begin{theorem}\label{T5}
\em{ If $f\in C[0,\infty)$ is convex then we have
 \begin{eqnarray*}
 f(x)\leq S_{n+1}(f;x)\leq S_n(f;x),~~~~~~~~~~~x\geq 0,~~~n\geq 1,
\end{eqnarray*}  
as well as if  $f$ is increasing (decreasing), then $S_n(f;x)$ is increasing (decreasing).}
\end{theorem}
\begin{theorem}\em{
Let $f\in C[0,\infty)$ is decreasing and convex for any $x\in[0,\infty),$ then there exists $N\in\mathbb{N}$ such that 
\begin{eqnarray*}
S_{n,a}^*(f;x)\geq  S_{n+1,a}^*(f;x)\geq f(x),~~~~~~ n\geq N.
\end{eqnarray*} }
\end{theorem}
\begin{proof}
Before passing to the direct proof of the above theorem, we shall express the given operators (\ref{s}) in form of  Sz$\acute{\text{a}}$sz-Mirakjan operators. We can write $S_{n,a}^*(f;x)=S_n(f;l_n)$, where  
\begin{eqnarray*}
l_n=(S_n(a^t;x))^{-1}\circ a^x.
\end{eqnarray*}
Since here $a^x,\; 0\leq x<\infty,\; a>1$ is a convex function and using Theorem \ref{T5},  then we have
\begin{eqnarray*}
 S_{n}a^t&\geq &  S_{n+1}a^t,\\
(S_{n+1}a^t)^{-1} &\geq &  (S_{n}a^t)^{-1},\\
l_{n+1}=(S_{n+1}a^t)^{-1}\circ a^x &\geq &  (S_{n}a^t)^{-1}\circ a^x=l_n.
\end{eqnarray*}
i.e.,
\begin{eqnarray}\label{V6}
l_n(x)\leq l_{n+1}(x)
\end{eqnarray}
Since $f\in C[0,\infty)$ is convex and decreasing then we have 
\begin{eqnarray}\label{V7}
f(x)\leq S_n(f;x)\leq S_{n+1}(f;x)
\end{eqnarray}
Using (\ref{V6}), we have 
\begin{eqnarray*}
S_{n,a}^*(f;x)-S_{n+1,a}^*(f;x)&=&(S_n f)\circ(l_{n}(x))-(S_{n+1}f)\circ(l_{n+1}(x))\\
&=&(S_{n}f)\circ(l_n(x))-(S_{n+1}f)\circ(l_{n}(x))+(S_{n+1}f)\circ(l_{n+1}(x))\\&&-(S_{n+1}f)\circ(l_{n}(x))\\ &\geq0.&
\end{eqnarray*}
And hence 
\begin{eqnarray}\label{V8}
S_{n,a}^*(f;x)\geq  S_{n+1,a}^*(f;x),\;\;\;\;\;\;\ \forall\; x\in[0,\infty),\; a>1.
\end{eqnarray}
Since  last inequality of the theorem is a consequence of  inequality (\ref{V8}) and Theorem \ref{T2}.
\end{proof}
\begin{theorem}\em{
Let $f\in C[0,\infty)$ be an increasing and $(1,a^x)$convex with respect to  $a^x$, $a>1$. Then
\begin{eqnarray*}
f(x)\leq S_{n,a}^*(f;x)\leq S_n(f;x),~~~~~\forall\; x\in[0,\infty).
\end{eqnarray*} }
\end{theorem} 
\begin{proof}
With the help of Theorem (3) and Remark (2) of Ziegler \cite{ZZ} and since the function $f$ is $(1,a^x)$ convex with respect to   $a^x,\; a>1,\; x\geq0$ and also $a^x$ is a convex function, then we have 
\begin{eqnarray*}
f(x)\leq S_{n,a}^*(f;x),\;\;\;\;\ x\in[0,\infty),\;a>1,\; n\geq 1,
\end{eqnarray*}
but also $S_n(a^t)\geq a^x$ as $a^x$ is convex and  $(S_n(a^t))^{-1}$ is increasing, so we have 
\begin{eqnarray}\label{V9}
(S_n(a^t))^{-1}\circ S_n(a^t) &\geq & (S_n(a^t))^{-1}\circ a^x,\nonumber\\ \Rightarrow~~~~~~~~~~~~~~~~~~~~~~~~~~~~~~~~ x &\geq & \left((S_n(a^t))^{-1}\circ a^x\right)(x),
\end{eqnarray}
since, $l_n=(S_n(a^t))^{-1}\circ a^x.$\\

Then from (\ref{V9}), it  directly follows 
\begin{eqnarray*}
S_{n,a}^*(f;x)\leq  S_{n}(f;x).
\end{eqnarray*}
\end{proof}
\section{An Extension}
In a further investigation we obtain approximations by linear positive operators using  new  Kantorovich- type operators in the space of integrable function. The Kantorovich type operators of the operators (\ref{s}) are given as bellow:  

\begin{eqnarray}\label{NO}
\tilde{S}_{n,a}^*(f;x)=n\sum\limits_{k=0}^{\infty}a^{\left(\frac{-x}{-1+a^{\frac{1}{n}}}\right)}\frac{(x\log{a})^k}{(-1+a^{\frac{1}{n}})^kk!}\int\limits_{\frac{k}{n}}^{\frac{k+1}{n}}f(t)\,dt.
\end{eqnarray}

Here we get the following identities\label{I3} as given bellow. 
\begin{eqnarray*}
1.~\tilde{S}_{n,a}^*(e_0;x)&=& 1,\\
2.~\tilde{S}_{n,a}^*(e_1;x)&=& \frac{1}{2n}+\frac{x\log{a}}{\left(-1+a^{\frac{1}{n}}\right)n},\\
3.~\tilde{S}_{n,a}^*(e_2;x)&=& \frac{1}{3n^2}+\frac{2x\log{a}}{\left(-1+a^{\frac{1}{n}}\right)n^2}+\frac{x^2(\log{a})^2}{\left(-1+a^{\frac{1}{n}}\right)^2n^2}.
\end{eqnarray*}
With regard of the above integral operators (\ref{NO}), we raise the problem to investigate
their convergence in $L_p-$spaces.\\

Here, by the above identities, we get 
\begin{eqnarray*}
\underset{n\rightarrow\infty}{\lim} \tilde{S}_{n,a}^*(e_i)=e_i,\;\;\;\;\;\ e_i=x^i,\; i=0,1,2.
\end{eqnarray*}

\end{document}